\documentclass[9pt,reqno]{article}
\usepackage[english]{babel}
\usepackage{tikz}
\usepackage{caption} % Include the caption package
\usepackage{amsmath,amsthm}
\usepackage{amsfonts}
\usepackage{enumerate}
\usepackage{float}
\usepackage{pgfplots}
\usepackage{xcolor}
\usepackage{natbib}

\usepackage{booktabs} 
\usepackage{graphicx}
\usepackage{subcaption}
\usepackage[colorlinks = true,urlcolor = blue,citecolor = blue,breaklinks]{hyperref}

\newtheorem{lemma}{Lemma}[section]

\newtheorem{proposition}{Proposition}[section]

\theoremstyle{definition}

\theoremstyle{remark}

\usepackage{geometry}
\begin{document}
{\Large \textbf{A short proof of Feige's conjecture for identically distributed random variables}}

\bigskip

\noindent

\bigskip\textbf{Mart\'{\i}n Egozcue} \footnote{Email: egozcuemj@gmail.com} \newline\textit{University of Montevideo and UCU Sede Punta del Este, Uruguay} \bigskip

\textbf{Luis Fuentes Garc\'{\i}a} \newline\textit{Departament of Mathematics, University of Coru\~{n}a,
15071 A Coru\~{n}a, Spain.}

\bigskip Abstract: \noindent In this note, we provide a short proof of Feige's conjecture for identically distributed random variables.

\section{Introduction}

In this note, we study a conjecture posits in \cite{fei}, which states that 

\[
\mathbf{P}(X_1+X_2+...+X_n<1+n)\geq 1/e,
\]
for independent, non-negative random variables and mean equal to one.

Many advances to prove this inequality have been made see  \cite{fei}; \cite{elt}; \cite{he}; \cite{gar}; \cite{guo}; \cite{alq}. \cite{fei} showed that to prove this general conjecture, one only needs to prove it for the case when each variable $X_i$ has the entire probability mass distributed on 0 and another point such that $\mathbb{E}[X_i] = 1$ (see also, \cite{bur}). 

We follow this approach and we show that the conjecture holds for equally distributed random variables. Our proof relies in some basic techniques used in the works by \cite{rigo} and \cite{bur}. 

\begin{proposition}\label{p1} Let $X_i$ be independent random variables with $\mathbf{P}(X_i=x)=1/x$ and $\mathbf{P}(X_i=0)=1-1/x,$ for all $i=1,2,...,n$ and $x>1$. Then, 

\[
\mathbf{P}(X_1+X_2+...+X_n<1+n)\geq 1/e.
\]
\end{proposition}

\begin{proof}   
Since we have assumed that $P(X_i = x) = \frac{1}{x} = p$, $P(X_i = 0) = 1 - p$, then:

\[
f(p) = P(S_n < 1 + n) = \sum_{k \in \mathbb{Z}_{\geq 0}, \, k/p < 1 + n} \binom{n}{k} p^k (1-p)^{n-k}.
\]

The condition $k/p < 1 + n$ is equivalent to $k < p(1+n)$, and for $k$ an integer, this means $k \leq \lceil p(1+n) \rceil - 1$, where $\lceil z \rceil$ is the ceiling function. Thus:
\[
f(p) = P(S_n < 1 + n) = \sum_{k=0}^{\lceil (1+n)p \rceil - 1} \binom{n}{k} p^k (1-p)^{n-k}.
\]

The discontinuity in $f(p)$ is therefore a consequence of the number of summands increasing or decreasing, a condition that occurs precisely when $(1+n)p$ becomes an integer. This behavior is visualized in Figure 1 as a distinct "sawtooth" pattern.

\begin{figure}[H]
    \centering
    \includegraphics[width=0.50\textwidth]{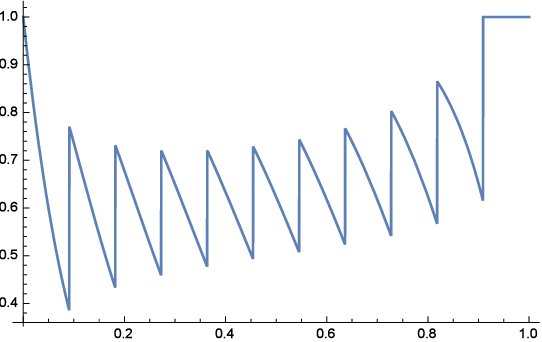}
    \caption{Plot of $f(p)$ for $n=10$}
    \label{fig:both_figures}
\end{figure}

The function $f(p)$, with $p\in (0,1)$, is continuous and differentiable in each interval where $\left(\frac{k}{1+n}, \frac{k+1}{1+n}\right],$ $0\leq k\leq n$.

In each of these intervals, $f(p)$ is decreasing. That is, if we consider:
\[
F_{n,m}(p) = \sum_{k=0}^{m} \binom{n}{k} p^k (1-p)^{n-k},
\]
it is decreasing in $p$.\footnote{\cite{elt} investigated the case where all random variables are identically distributed and take only integer values. Furthermore, through private communication, he shared with us an interesting observation based on a result by \cite{kaa}, which associates the median and mean of $S_n$. However, his result is not necessary for our proof.}

This can be proven as follows.\footnote{The next derivation resembles a generalization solution to John Smith's problem, see \cite{chau} and \cite{lov}.}  We have
\[
\left( p^k (1-p)^{n-k} \right)' = k p^{k-1} (1-p)^{n-k} - (n-k) p^k (1-p)^{n-k-1} = p^{k-1} (1-p)^{n-k-1} \left( k - n p \right),
\]

and this value is negative (or zero) if $k \leq n p$, and positive if $k > n p$.

This implies that, for a fixed $p,$
\[
F'_{n,m}(p) \leq 0 \quad \text{for } m \leq m_0 = \lfloor n p \rfloor.
\]

For $m > \lfloor n p \rfloor$, from $F'_{n,m_0}(p)$ to $F'_{n,m}(p)$, positive terms are added. Therefore:
\[
F'_{n,m_0}(p) < F'_{n,m_0+1}(p) < F'_{n,m_0+2}(p) < \dots < F'_{n,n}(p).
\]

Additionally, we know that $F_{n,n}(p) = 1$ is constant, and hence $F'_{n,n}(p) = 0$. Thus,
\[
F'_{n,m_0}(p) < F'_{n,m_0+1}(p) < F'_{n,m_0+2}(p) < \dots < F'_{n,n}(p) = 0,
\]

and effectively $F'_{n,m}(p) < 0$ for any $m$, $0 \leq m < n$.

Therefore, the minimum is achieved when $p$ is the final value of one of those intervals, that is, when $p = \frac{m}{n+1}$ for some integer $m< n+1$. In these cases, the expression for the probability of study is
\[
h(n, m) = P(S_n < 1+n) = \frac{1}{(n+1)^n} \sum_{k=0}^{m-1} \binom{n}{k} m^k (n+1-m)^{n-k}.
\]

Notice that we can express this probability by using the regularized incomplete beta function,\footnote{We are using the notation $B(z;a,b)=\int_o^zt^{a-1}(1-t)^{b-1}dt.$} with \( p = \frac{m}{n+1} \), as follows

\[
h(n,m) = (n-m+1)\binom{n}{m-1}B(1-p; n-m+1, m).
\]

It is straightforward to see that \( h(n,1) < h(n,2) \). Next, we shall prove that 
\begin{equation}
    d(m) = h(n, m+1)-h(n, m)\geq 0 \quad \text{for all } m > 1. \label{ineq1}
\end{equation}

But first note that $d(m)$ is symmetric around $m=n/2$ (see Lemma \ref{aux2} in the Appendix). Therefore, if we prove that $d(m)\geq0$ for either $m\geq n/2$ or for $m\leq n/2$ we prove that $d(m)\geq0$ for all $m.$

Inequality (\ref{ineq1}) is equivalent to proving that
\[
(n-m)\binom{n}{m}B(1-q; n-m, m+1) \geq (n-m+1)\binom{n}{m-1}B(1-p; n-m+1, m),
\]
where \( q = \frac{m+1}{n+1} \).

Since,
\[
(n-m+1)\binom{n}{m-1} = m\binom{n}{m},
\]

the inequality to prove simplifies to
\[
(n-m)B(1-q; n-m, m+1) \geq mB(1-p; n-m+1, m).
\]

Notice that,
\[
(n-m)\int_0^{1-q}t^{n-m-1}(1-t)^{m}dt \geq m\int_0^{1-p}t^{n-m}(1-t)^{m-1} dt,
\]

is equivalent to

\[
\int_{0}^{1-q} \left[ \frac{d}{dt}\Bigl(t^{n-m}\Bigr) (1-t)^m + t^{n-m} \frac{d}{dt}\Bigl((1-t)^m\Bigr) \right] dt
\ge m \int_{1-q}^{1-p} t^{n-m}(1-t)^{m-1} \, dt.
\]

Equivalently
\[
q^{m}(1-q)^{n-m} = \left( \frac{m+1}{n+1} \right)^m \left( \frac{n-m}{n+1} \right)^{n-m} \geq m \int_{1-q}^{1-p} t^{n-m} (1-t)^{m-1} \, dt.
\]

Notice that the function \( g(t) = t^{n-m}(1-t)^{m-1} \) first derivative is equal to 
\[g'(t) = g(t) \left( \frac{n-m}{t} - \frac{m-1}{1-t} \right).\]

One can check that $g(t)$ in the range \([1-q, 1-p]\) achieves its maximum at $t = 1-p$ when $m\leq \frac{n+1}{2}$ or $t=\frac{n-m}{n-1}$ in the other cases.

We start when the maximum is achieved at $t=1-p$ and we will show that the inequality holds for all $m\geq n/2.$

Owing that,
\[
m\int_{1-q}^{1-p}(1-p)^{n-m}p^{m-1} dt = m(1-p)^{n-m}p^{m-1} \frac{1}{n+1} \geq m \int_{1-q}^{1-p}t^{n-m}(1-t)^{m-1} dt.
\]

Hence, we need to analyze whether

\[
q^m(1-q)^{n-m}\geq \frac{m}{n+1}(1-p)^{n-m}p^{m-1}=\frac{(n+1-m)^{n-m}m^{m}}{(n+1)^{n}}.
\]

This involves examining the series of comparable inequalities:

\begin{align*}
(m+1)^m(n-m)^{n-m}&\geq (n+1-m)^{n-m}m^m,\\
\left(1+\frac{1}{m}\right)^m&\geq \left(1+\frac1{n-m}\right)^{n-m},\\
w(m)&\geq w(n-m).\\
\end{align*}

Since $w(x)$ is  strictly increasing (see Lemma \ref{aux1} in the Appendix), this is equivalent to $m\geq n-m$, that is, $m\geq n/2$.

Now the second case, if the maximum is achieved at $t=\frac{n-m}{n-1}.$ The the inequality to be proved is 

\[
q^m (1-q)^{n-m} \geq \frac{m}{n+1} \left(\frac{n-m}{n-1}\right)^{n-m} \left(1 - \frac{n-m}{n-1}\right)^{m-1} = \frac{m (n-m)^{n-m} (m-1)^{m-1}}{(n+1) (n-1)^{n-1}}.
\]

After some algebra this is equivalent to prove that 

\begin{equation}\label{ine1}
(m+1)^m(n-1)^{n-1}\geq m(m-1)^{m-1}(n+1)^{n-1}.
\end{equation}

Equivalently, 

\begin{equation}\label{ine2}
\underbrace{\frac{m(m-1)^{m-1}}{(m+1)^m}}_{b(m)} \leq \frac{(n-1)^{n-1}}{(n+1)^{n-1}}.
\end{equation}

Then,  
\[
\frac{b(m)}{b(m+1)}=\frac{m(m+2)^{m+1}(m-1)^{m-1}}{(m+1)m^m(m+1)^m}=\frac{(m+2)^{m+1}(m-1)^{m-1}}{m^{m-1}(m+1)^{m+1}}=\frac{w(m+1)}{w(m-1)}>1,\quad (\hbox{by Lemma }\ref{aux1}).
\]
So $b(m)$ is decreasing and therefore it is sufficient to prove (\ref{ine2}) and then (\ref{ine1}) for $m=(n+1)/2$; equivalently for $n=2m-1$. In this case, the sequences of equivalent inequalities are:

\begin{align*}
(m+1)^m(2(m-1))^{2m-2}&\geq m(m-1)^{m-1}(2m)^{2m-2},\\
(m+1)^m(m-1)^{2m-2}&\geq m^{2m-1}(m-1)^{m-1},\\
(m+1)^m(m-1)^{m-1}&\geq m^{2m-1},\\
\left(1+\frac{1}{m}\right)^m&\geq \left(1+\frac{1}{m-1}\right)^{m-1},\\
w(m)&\geq w(m-1).
\end{align*}

And the last inequality follows again from Lemma \ref{aux1}.

The analysis of these two cases proves the inequality \( h(n,1) < h(n,m) \) for all \( m > 1 \). Therefore, we need to minimize

\[
h(n,1) = \left(\frac{n}{n+1}\right)^n.
\]

It can be shown that \( h(n,1) \) is a decreasing function of \( n \). Therefore, the minimum is achieved when \( n \) tends to infinity, which yields the desired result.
 \end{proof}

\section*{Acknowledgements} 
We are grateful to Prof. Robbert Fokkink for pointing out the association between the topic of this work and John Smith's problem. We extend our gratitude to John H. Elton for his insightful comments.

\section*{Appendix}

\begin{lemma}\label{aux2}
Let, $d(m)=h(n,m+1)-h(n,m),$ where $h(n,m)=\displaystyle \sum_{k=0}^{m-1} \binom{n}{k} m^k (n+1-m)^{n-k}$. Then we have, 
$$d(m)=d(n-m).$$
\end{lemma}
\begin{proof}
We have that:
\begin{align*}
h(n,m) &= \sum_{k=0}^{m-1} \binom{n}{k} m^k (n+1-m)^{n-k} \\
       &= (n+1)^n - \sum_{k=m}^{n} \binom{n}{k} m^k (n+1-m)^{n-k} \\
       &= (n+1)^n - \sum_{k=m}^{n} \binom{n}{n-k} m^k (n+1-m)^{n-k}. \\
\intertext{By making the change \( n-k = s \), we get:}
h(n,m) &= (n+1)^n - \sum_{s=0}^{n-m} \binom{n}{s} m^{n-s} (n+1-m)^s \\
       &= (n+1)^n - \sum_{s=0}^{n+1-m-1} \binom{n}{s} (n+1-(n+1-m))^{n-s} (n+1-m)^s \\
       &= (n+1)^n - h(n,n+1-m).  \\
\end{align*}
From this:
$$
d(m) = h(n,m+1) - h(n,m)= h(n,n+1-m) - h(n,n-m)= d(n-m).
$$
\end{proof}

\begin{lemma}\label{aux1}
The function $w(x)=\left(1+\frac{1}{x}\right)^x$ is strictly increasing for $x>0$.
\end{lemma}
\begin{proof}
The first derivative is $w'(x)=w(x)\left(\log\left(1+\frac{1}{x}\right)-\frac{1}{1+x}\right)$. Notice that, $\log(1+z)> \frac{z}{1+z}$ for any $z>0$, hence:
$$
\log\left(1+\frac{1}{x}\right)>\frac{1/x}{1+1/x}=\frac{1}{1+x}.
$$
\end{proof}

\end{document}